\documentclass[a4paper,12pt]{article}

\usepackage[left=31.7mm,right=31.7mm,top=25.4mm,bottom=25.4mm]{geometry}    
\usepackage{color}                       

\usepackage{amsmath,amsthm,amsfonts,amssymb,bm,mathtools} 
\usepackage{graphicx,psfrag,subfigure}                    
\usepackage{mathrsfs}                           

\allowdisplaybreaks      
\numberwithin{equation}{section}
\numberwithin{figure}{section}
\numberwithin{table}{section}

\renewcommand{\d}{\mathrm{d}}
\renewcommand{\i}{\mathrm{i}}
\newcommand{\e}{\mathrm{e}}
\newcommand{\R}{\mathbb{R}}

\newcommand{\N}{\mathbb{N}}
\newcommand{\daoshu}[2]{\dfrac{\d #1}{\d #2}}
\newcommand{\SN}{\mathcal{S}_\mathcal{N}}
\newcommand{\SL}{\mathcal{S}_\mathcal{L}}
\newcommand{\SNh}{\mathcal{S}_\mathcal{N}^h}
\newcommand{\SLh}{\mathcal{S}_\mathcal{L}^h}
\newcommand{\Hp}{H_\text{\rm per}}
\newcommand{\Lp}{L_\text{\rm per}}
\newtheorem{example}{Example}[section]
\newtheorem{theorem}{Theorem}[section]
\newtheorem{lemma}{Lemma}[section]
\newtheorem{proposition}{Proposition}[section]
\newtheorem*{remark}{Remark}

\begin{document} 

\title{Convergence of a Fast Explicit Operator Splitting Method for the Molecular Beam Epitaxy Model   
\date{}
\author{Xiao Li\footnotemark[2]\and Zhonghua Qiao\footnotemark[3]\and Hui Zhang\footnotemark[4]}

\renewcommand{\thefootnote}{\fnsymbol{footnote}}
\footnotetext[2]
{School of Mathematical Sciences, Beijing Normal University, Beijing, P.~R.~China, 100875 (lixiao1228@163.com).}
\footnotetext[3]
{Department of Applied Mathematics, The Hong Kong Polytechnic University, Hung Hom, Kowloon, Hong Kong (zhonghua.qiao@polyu.edu.hk).}
\footnotetext[4]
{Laboratory of Mathematics and Complex Systems, Ministry of Education and School of Mathematical Sciences,
Beijing Normal University, Beijing, P.~R.~China, 100875 (hzhang@bnu.edu.cn).}
\renewcommand{\thefootnote}{\arabic{footnote}}

\maketitle                                       

\begin{abstract}
A fast explicit operator splitting (FEOS) method for the molecular beam epitaxy model has been presented in
[Cheng, et al., Fast and stable explicit operator splitting methods for phase-field models, J. Comput. Phys., submitted].
The original problem is split into linear and nonlinear subproblems.
For the linear part, the pseudo-spectral method is adopted;
for the nonlinear part, a 33-point difference scheme is constructed.
Here, we give a compact center-difference scheme involving fewer points for the nonlinear subproblem.
Besides, we analyze the convergence rate of the algorithm.
The global error order $\mathcal{O}(\tau^2+h^4)$ in discrete $L^2$-norm is proved theoretically and verified numerically.
Some numerical experiments show the robustness of the algorithm for small coefficients of the fourth-order term for the one-dimensional case.
Besides, coarsening dynamics are simulated in large domains and the $1/3$ power laws are observed for the two-dimensional case.
\end{abstract}


\section{Introduction}

Recently, the molecular beam epitaxy (MBE) has become an important technique for the growth of thin films.
By using the MBE technique, it is possible to grow high-quality crystalline materials
and form structures with high precision in the vertical direction \cite{HeSi89}.
There has been a large amount of research interest in the dynamics of the MBE growth.
Different kinds of models have been developed to describe the growth evolution,
including atomistic models, continuum models, and hybrid models \cite{EvThBa06}.
In our work, we are interested in the continuum models for the evolution of the MBE growth.
The evolution is governed by the following nonlinear partial differential equation:
\begin{equation}
\label{MBE2D_model1}
u_t=\nabla\cdot[(|\nabla u|^2-1)\nabla u]-\delta\Delta^2u,\quad(x,y)\in\Omega,\ t\in(0,T],
\end{equation}
where $\delta>0$ is a constant, $\Omega=(0,2L)^2$ with $L>0$, and $u:\Omega\times(0,\infty)\to\R$ is an $\Omega$-periodic scaled height function
equipped with the initial data $$u(x,y,0)=u_0(x,y),\quad(x,y)\in\overline{\Omega}.$$
The fourth-order term models the surface diffusion and the nonlinear second-order term models the Ehrlich-Schwoebel effect \cite{EhHu66,MoGo00,ScSh66}.
The equation (\ref{MBE2D_model1}) is the gradient flow with respect to the $L^2(\Omega)$ inner product of the energy functional
\begin{equation}
\label{MBE2D_energy1}
E(u)=\int_\Omega\Big(\frac{1}{4}(|\nabla u|^2-1)^2+\frac{\delta}{2}|\Delta u|^2\Big)\,\d x\d y.
\end{equation}
With the periodic boundary condition, it is easy to show that the energy $E$ decreases with respect to the time.
For the coarsening dynamics governed by (\ref{MBE2D_model1}),
the exponents measured experimentally are $1/3$, which is observed in numerical simulations of the MBE growth \cite{ChKuQuTa15,QiZhTa11,ShWaWaWi12,XuTa06}
and analyzed quantitively by introducing a kinetic scaling theory \cite{LiLi03}.

There have been many theoretical and numerical studies on the MBE models.
The well-posedness and regularity of the initial-boundary-value problem of the model (\ref{MBE2D_model1})
are studied in \cite{LiLi03} using the Galerkin approximation method.
For the MBE simulations, a large computational domain is necessary
in order to make the effect of periodicity assumption as small as possible and to collect enough statistical information
such as mean surface height and width of the pyramid-like structures.
Besides, a sufficiently long integration time is necessary
in order to detect the epitaxy growth behaviors and to reach the physical scaling regime.
To carry out numerical simulations with large time and large computational domain,
highly stable and accurate numerical schemes are required.
The equation (\ref{MBE2D_model1}) is highly nonlinear with a small surface diffusion parameter $\delta$,
which makes it difficult to design an effective numerical scheme.
In \cite{QiSuZh12}, two stable and convergent linearized difference schemes are derived by using the method of reduction of order \cite{Sun09}.
The convergence rates are $\mathcal{O}(\tau+h^2)$ and $\mathcal{O}(\tau^2+h^2)$ in discrete $L^2$-norm, respectively.
Both the nonlinear part and the diffusion term are treated explicitly there.
In \cite{QiZhTa11}, two unconditionally energy stable difference schemes are presented.
These two schemes are second-order convergent in time and nonlinear.
Because of the unconditional stability, an adaptive time-stepping strategy is purposed there.
In \cite{ShWaWaWi12,WaWaWi10},
the first- and second-order (in time) convex splitting schemes are constructed under the framework exploited by Eyre \cite{Eyr98}.
Still, both the two schemes are nonlinear and unconditionally energy stable.
The similar technique has been used extensively on different phase field models, e.g.,
the phase field crystal model \cite{WiWaLo09}, a diffusive interface model with Peng-Robinson equation of state \cite{QiSu14}, etc.
In \cite{XuTa06}, the authors introduce an implicit-explicit scheme combined with Fourier pseudo-spectral approach,
where the nonlinear term is treated explicitly and the fourth-order term implicitly.
To guarantee the stability, they add an extra artificial term consistent with the truncated errors in time.
However, the condition, under which the energy stability can be obtained without any restriction on time step, depends on the unknown numerical solutions.
In \cite{QiTaXi15}, a mixed finite element method with Crank-Nicolson time-stepping scheme is presented
and the energy laws are proved for both semi- and fully-discrete form of the scheme.

In \cite{ChKuQuTa15}, a fast explicit operator splitting (FEOS) method based on the Strang splitting schemes \cite{Str68}
is constructed to simulate the MBE equations for both one- and two-dimensional cases.
The main idea of the method is to split the original equation (\ref{MBE2D_model1}) into nonlinear and linear parts
whose exact solution operators are denoted by $\SN$ and $\SL$,
and then to evolve one splitting step (from $t$ to $t+\tau$) via three substeps:
$$u(x,y,t+\tau)=\SL\Big(\frac{\tau}{2}\Big)\SN(\tau)\SL\Big(\frac{\tau}{2}\Big)u(x,y,t).$$
A similar strategy has also been used to solve the phase field crystal equation in \cite{LeShLe15}.
In \cite{ChKuQuTa15},
the nonlinear part is solved by the 33-point center-difference scheme combined with the large stability domain explicit Runge-Kutta solver,
and the linear one is solved by the pseudo-spectral method.
Their numerical experiments indicate that the proper constant time step should be $\tau=\delta/100$.
In addition, the FEOS method has also been successfully utilized on
the convection-diffusion equations equations \cite{ChKu09,ChKuPe05,ChKuPe09} and the modified Buckley-Leverett equations \cite{KaKuQuWa14}.
It is capable to achieve a reliable numerical solutions in an efficient manner, that is, only few splitting steps are preformed \cite{ChKuPe09}.

In our work, we concentrate mainly on the convergence analysis of the FEOS method for the MBE equation in the two-dimensional case.
The main issue, which is different from that in \cite{ChKuQuTa15}, consists of three aspects.
First, we discretize the nonlinear part by a 25-point center-difference scheme in space
and the explicit strong stability preserving Runge-Kutta method in time,
and combine the so-called ``frozen coefficient'' technique with the Fourier analysis method to derive a constraint on the time step for the stability.
Second, we analyze the convergence of the entire algorithm.
The global discrete $L^2$-error consists of the truncation errors from the splitting, the nonlinear and linear schemes, respectively.
Third, we carry out some numerical experiments to verify the convergence rate,
and test the robustness of the algorithm with small $\delta$ in the one-dimensional case.
Numerical experiments suggest that the time step can be set as $\tau=\delta/10$ using our algorithm.
This result is a little better than that in \cite{ChKuQuTa15},
because the difference scheme for nonlinear part involves fewer points, which may loosen the restriction on the time step.
Besides, we consider the two-dimensional coarsening dynamics
to observe the $-1/3$ power law of the energy and the $1/3$ power law of the mean height.

The organization of this paper is as follows.
In Section 2, we present the fast explicit operator splitting method for the two-dimensional MBE equation,
and give a sufficient condition for the stability of the algorithm here.
In Section 3, the discrete $L^2$-error estimate of the FEOS method is shown both theoretically and numerically.
Further numerical experiments are carried out and the power law for the coarsening dynamics is observed in Section 4.
Some concluding remarks are given in Section 5.

\section{Fast explicit operator splitting method}

Here we present the algorithm developed in \cite{ChKuQuTa15}
where the nonlinear and linear parts are approximated by different methods,
and construct a more compact difference scheme for the nonlinear part.

\subsection{Splitting strategy}\label{ss_split_strategy}

In \cite{ChKuQuTa15}, the equation (\ref{MBE2D_model1}) is split into the nonlinear part
\begin{equation}
\label{MBE2D_nonlinear}
u_t=\nabla\cdot(|\nabla u|^2\nabla u),
\end{equation}
and the linear part
\begin{equation}
\label{MBE2D_linear}
u_t=-\Delta u-\delta\Delta^2u,
\end{equation}
whose exact solution operators are denoted by $\SN$ and $\SL$, respectively.
Introducing a splitting time step $\tau$, the solution of the equation (\ref{MBE2D_model1})
is resolved from $t$ to $t+\tau$ via the Strang splitting method \cite{Str68} consisting of three substeps:
\begin{equation}
\label{StrangOS_LNL}
u(x,y,t+\tau)=\SL\Big(\frac{\tau}{2}\Big)\SN(\tau)\SL\Big(\frac{\tau}{2}\Big)u(x,y,t).
\end{equation}
In general, if all the solutions involved in the three-step splitting scheme (\ref{StrangOS_LNL}) are smooth,
the operator splitting method is second-order accurate \cite{Str68}.

For the nonlinear subproblem (\ref{MBE2D_nonlinear}), the solution is $L^2$-stable with respect to the initial data,
which is described precisely by the following proposition.

\begin{proposition}
\label{prop_stability_SN}
Assuming that $u_0,v_0\in\Hp^2(\Omega)$, we have
$$\|\SN(t)u_0-\SN(t)v_0\|_{L^2(\Omega)}\le\|u_0-v_0\|_{L^2(\Omega)},\quad\forall t>0,$$
where $\Hp^2(\Omega)=\{u\in H^2(\Omega)\,|\,u~\text{is $\Omega$-periodic}\}$.
\end{proposition}

\begin{proof}
Set $u(x,y,t)$ and $v(x,y,t)$ to be the solutions of (\ref{MBE2D_nonlinear})
with the initial data $u(\cdot,\cdot,0)=u_0$ and $v(\cdot,\cdot,0)=v_0$, respectively.
Let $w=u-v$, then we have
$$w_t=\nabla\cdot(|\nabla u|^2\nabla u-|\nabla v|^2\nabla v)=\frac{1}{2}\nabla\cdot\big((|\nabla u|^2+|\nabla v|^2+|\nabla u+\nabla v|^2)\nabla w\big).$$
Taking the inner-product with $w$ and noting the periodicity, we obtain
$$\daoshu{}{t}\|w(t)\|_{L^2(\Omega)}^2=-\int_\Omega(|\nabla u|^2+|\nabla v|^2+|\nabla u+\nabla v|^2)|\nabla w|^2\,\d x\d y\le 0,$$
which leads to
$$\|w(t)\|_{L^2(\Omega)}\le\|w(0)\|_{L^2(\Omega)},\quad\forall t>0,$$
that is,$$\|u(t)-v(t)\|_{L^2(\Omega)}\le\|u(0)-v(0)\|_{L^2(\Omega)},\quad\forall t>0,$$
which completes the proof.\qquad
\end{proof}

In practice, the exact solution operators $\SN$ and $\SL$ are to be replaced by their numerical approximations.
In the following two subsections, we present the numerical methods given in \cite{ChKuQuTa15},
while the algorithm for the nonlinear part is a little different.

\subsection{Center-difference scheme for the equation (\ref{MBE2D_nonlinear})}\label{ss_nonlinear}

Using the method of lines, the nonlinear subproblem (\ref{MBE2D_nonlinear}) can be reduced to a system of ODEs,
which can be efficiently and accurately integrated by a stable explicit ODE solver.
Here we adopt the fourth-order-difference to discrete the space,
and choose the third-order strong stability preserving Runge-Kutta (SSP-RK3) method \cite{GoShTa01} as the ODE solver.

Introducing a spatial scale $h=2L/J$, where $J=2N$ is a positive even integer,
the grid nodes are defined as $(x_j,y_k)=(jh,kh)$, $j,k=1,2,\dots,J$.
The fourth-order semi-discrete scheme for (\ref{MBE2D_nonlinear}) can be written as \cite{KuLe00}
\begin{align}
\daoshu{u_{j,k}(t)}{t} & =\frac{-F_{j+2,j,k}(t)+8F_{j+1,j,k}(t)-8F_{j-1,j,k}(t)+F_{j-2,j,k}(t)}{12h}\nonumber\\
& \qquad +\frac{-G_{j,k,k+2}(t)+8G_{j,k,k+1}(t)-8G_{j,k,k-1}(t)+G_{j,k,k-2}(t)}{12h},\label{nonlinear_ODE}
\end{align}
where
$$F_{j+\ell,j,k}=F((u_x)_{j+\ell,j,k},(u_y)_{j+\ell,j,k}),\ G_{j,k,k+\ell}=G((u_x)_{j,k,k+\ell},(u_y)_{j,k,k+\ell}),\ \ell=\pm1,\pm2,$$
here $F(p,q)=(p^2+q^2)p$, $G(p,q)=(p^2+q^2)q$, and
\begin{subequations}
\label{nonlinear_ODE_uxuy}
\begin{align}
(u_x)_{j+2,j,k} & =\frac{25u_{j+2,k}-48u_{j+1,k}+36u_{j,k}-16u_{j-1,k}+3u_{j-2,k}}{12h},\\
(u_x)_{j+1,j,k} & =\frac{3u_{j+2,k}+10u_{j+1,k}-18u_{j,k}+6u_{j-1,k}-u_{j-2,k}}{12h},\\
(u_x)_{j-1,j,k} & =\frac{u_{j+2,k}-6u_{j+1,k}+18u_{j,k}-10u_{j-1,k}-3u_{j-2,k}}{12h},\\
(u_x)_{j-2,j,k} & =\frac{-3u_{j+2,k}+16u_{j+1,k}-36u_{j,k}+48u_{j-1,k}-25u_{j-2,k}}{12h},\\
(u_y)_{j+\ell,j,k} & =\frac{-u_{j+\ell,k+2}+8u_{j+\ell,k+1}-8u_{j+\ell,k-1}+u_{j+\ell,k-2}}{12h},\ \ell=\pm1,\pm2,\\
(u_x)_{j,k,k+\ell} & =\frac{-u_{j+2,k+\ell}+8u_{j+1,k+\ell}-8u_{j-1,k+\ell}+u_{j-2,k+\ell}}{12h},\ \ell=\pm1,\pm2,\\
(u_y)_{j,k,k+2} & =\frac{25u_{j,k+2}-48u_{j,k+1}+36u_{j,k}-16u_{j,k-1}+3u_{j,k-2}}{12h},\\
(u_y)_{j,k,k+1} & =\frac{3u_{j,k+2}+10u_{j,k+1}-18u_{j,k}+6u_{j,k-1}-u_{j,k-2}}{12h},\\
(u_y)_{j,k,k-1} & =\frac{u_{j,k+2}-6u_{j,k+1}+18u_{j,k}-10u_{j,k-1}-3u_{j,k-2}}{12h},\\
(u_y)_{j,k,k-2} & =\frac{-3u_{j,k+2}+16u_{j,k+1}-36u_{j,k}+48u_{j,k-1}-25u_{j,k-2}}{12h}.
\end{align}
\end{subequations}
The fully-discrete scheme for (\ref{MBE2D_nonlinear}) is obtained by applying the SSP-RK3 method \cite{GoShTa01}.
This completes the numerical approximation of the operator $\mathcal{S_N}$.
We notice that our scheme (\ref{nonlinear_ODE})--(\ref{nonlinear_ODE_uxuy}) is fourth-order in space,
which is same as the scheme (2.7)--(2.9) given in \cite{ChKuQuTa15}.
In addition, our scheme is more compact than the scheme in \cite{ChKuQuTa15},
since the former utilizes 25 points while the latter 33, as proposed in Fig. \ref{fig_gridpoints}.

\begin{figure}[!htp]
\centering
\subfigure[The scheme (\ref{nonlinear_ODE})--(\ref{nonlinear_ODE_uxuy}) here.]{\includegraphics[scale=1.2]{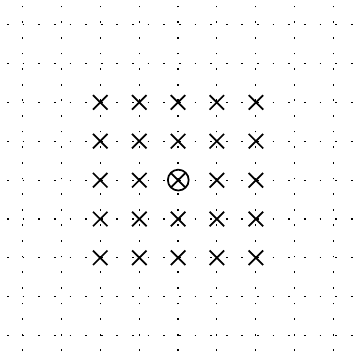}}\qquad\qquad
\subfigure[The scheme (2.7)--(2.9) in \cite{ChKuQuTa15}.]{\includegraphics[scale=1.2]{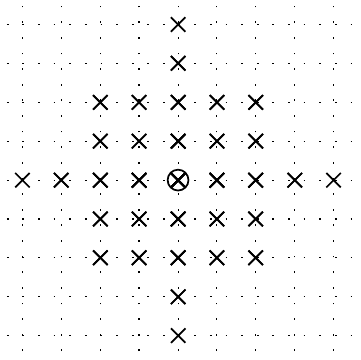}}
\caption{The ``$\times$'' represents the point involved in the scheme expanded at the ``$\circ$'' point.}\label{fig_gridpoints}
\end{figure}

According to the property of strong stability preserving,
the stability restriction of the SSP-RK3 method is identical to that of the forward Euler scheme.
We use the ``frozen coefficient'' strategy to analyze the stability of the forward Euler scheme
\begin{align}
\frac{u_{j,k}^{n+1}-u_{j,k}^n}{\tau} & =\frac{-F_{j+2,j,k}^n+8F_{j+1,j,k}^n-8F_{j-1,j,k}^n+F_{j-2,j,k}^n}{12h}\nonumber\\
& \qquad +\frac{-G_{j,k,k+2}^n+8G_{j,k,k+1}^n-8G_{j,k,k-1}^n+G_{j,k,k-2}^n}{12h},\label{nonlinear_forEuler}
\end{align}
where $F_{j+\ell,j,k}^n=F_{j+\ell,j}(t_n)$, $G_{j,k,k+\ell}^n=G_{j,k,k+\ell}(t_n)$, $\ell=\pm1,\pm2$, that is,
\begin{align*}
\frac{u_{j,k}^{n+1}-u_{j,k}^n}{\tau}
& =(|\nabla u|^2)_{j+2,j,k}^n\cdot\bigg[-\frac{(u_x)_{j+2,j,k}^n}{12h}\bigg]+(|\nabla u|^2)_{j+1,j,k}^n\cdot\bigg[\frac{2(u_x)_{j+1,j,k}^n}{3h}\bigg]\\
& \qquad +(|\nabla u|^2)_{j-1,j,k}^n\cdot\bigg[-\frac{2(u_x)_{j-1,j,k}^n}{3h}\bigg]+(|\nabla u|^2)_{j-2,j,k}^n\cdot\bigg[\frac{(u_x)_{j-2,j,k}^n}{12h}\bigg]\\
& \qquad +(|\nabla u|^2)_{j,k,k+2}^n\cdot\bigg[-\frac{(u_y)_{j,k,k+2}^n}{12h}\bigg]+(|\nabla u|^2)_{j,k,k+1}^n\cdot\bigg[\frac{2(u_y)_{j,k,k+1}^n}{3h}\bigg]\\
& \qquad +(|\nabla u|^2)_{j,k,k-1}^n\cdot\bigg[-\frac{2(u_y)_{j,k,k-1}^n}{3h}\bigg]+(|\nabla u|^2)_{j,k,k-2}^n\cdot\bigg[\frac{(u_y)_{j,k,k-2}^n}{12h}\bigg].
\end{align*}
It is observed that the terms $(u_x)_{j+\ell,j,k}^n$, $(u_y)_{j+\ell,j,k}^n$, $(u_x)_{j,k,k+\ell}^n$, $(u_y)_{j,k,k+\ell}^n$
approximate the values $u_x(x_{j+\ell},y_k,t_n)$, $u_y(x_{j+\ell},y_k,t_n)$, $u_x(x_j,y_{k+\ell},t_n)$, $u_y(x_j,y_{k+\ell},t_n)$
with the error $\mathcal{O}(h^4)$, respectively.
Freezing the prefactors of the square bracket terms by the constant
$$A=\max\{(|\nabla u|^2)_{j+\ell,j,k}^n,(|\nabla u|^2)_{j,k,k+\ell}^n:\ell=\pm1,\pm2\},$$
we obtain the following linear scheme:
\begin{align*}
\frac{u_{j,k}^{n+1}-u_{j,k}^n}{\tau} & =A\cdot\bigg(\frac{-u_{j+2,k}^n+16u_{j+1,k}^n-30u_{j,k}^n+16u_{j-1,k}^n-u_{j-2,k}^n}{12h^2}\\
& \qquad\qquad +\frac{-u_{j,k+2}^n+16u_{j,k+1}^n-30u_{j,k}^n+16u_{j,k-1}^n-u_{j,k-2}^n}{12h^2}\bigg),
\end{align*}
which can be transformed into the following form:
\begin{align}
u_{j,k}^{n+1} & =(1-5r)u_{j,k}^n+\frac{4r}{3}(u_{j+1,k}^n+u_{j-1,k}^n+u_{j,k+1}^n+u_{j,k-1}^n)\nonumber\\
& \qquad\qquad -\frac{r}{12}(u_{j+2,k}^n+u_{j-2,k}^n+u_{j,k+2}^n+u_{j,k-2}^n),\label{nonlinear_forEuler_expform}
\end{align}
where $r=A\tau/h^2$.
Using the Fourier analysis method, the symbol of the difference scheme (\ref{nonlinear_forEuler_expform}) is
$$\rho(\sigma_1,\sigma_2)=1-\frac{r}{3}[(1-\cos\sigma_1h)(7-\cos\sigma_1h)+(1-\cos\sigma_2h)(7-\cos\sigma_2h)].$$
Therefore, $|\rho(\sigma_1,\sigma_2)|\le 1$ if and only if
$$0\le r\le\frac{6}{(1-\cos\sigma_1h)(7-\cos\sigma_1h)+(1-\cos\sigma_2h)(7-\cos\sigma_2h)}.$$
As $(1-c)(7-c)\in[0,16]$ when $c\in[-1,1]$, we obtain $r\le\dfrac{3}{16}$, namely,
\begin{equation}
\label{nonlinear_stability}
\tau\le\frac{3}{16A}h^2.
\end{equation}
Obviously, this is a sufficient and unnecessary condition for the stability of the Euler scheme (\ref{nonlinear_forEuler}),
and thus, of the scheme (\ref{nonlinear_ODE}) combined with the SSP-RK3 solver.

\subsection{Pseudo-spectral method for equation (\ref{MBE2D_linear})}\label{ss_linear}

In \cite{ChKuQuTa15}, the equation (\ref{MBE2D_linear}) is solved by the pseudo-spectral method via the following procedure.
They first use the FFT algorithm to compute the discrete Fourier coefficients $\{\widetilde{u}_{pq}(t)\}$ from the point values $\{u_{j,k}(t)\}$.
Then they calculate $\widetilde{u}_{pq}(t+\tau)=\e^{\lambda_{pq}\tau}\widetilde{u}_{pq}(t)$, where
$$\lambda_{pq}=\frac{\pi^2(p^2+q^2)}{L^2}-\delta\bigg(\frac{\pi^2(p^2+q^2)}{L^2}\bigg)^2.$$
Finally they recover the point values of the solution at the new time level, $\{u_{j,k}(t+\tau)\}$,
from the discrete Fourier coefficients $\{\widetilde{u}_{pq}(t+\tau)\}$ using the inverse FFT algorithm.

For the self-consistent of our statement, here we give some formulas to be used in the next section.
For the continuous function $u(x,y,t)$, there exists the Fourier series in the complex form at time $t$:
\begin{equation}
\label{linear2D_series}
u(x,y,t)=\sum_{p=-\infty}^\infty\sum_{q=-\infty}^\infty\widehat{u}_{pq}(t)\e^{\frac{\i\pi}{L}(px+qy)},
\end{equation}
where the Fourier coefficients are given by
\begin{equation}
\label{linear2D_coef}
\widehat{u}_{pq}(t)=\frac{1}{4L^2}\int_\Omega u(x,y,t)\e^{-\frac{\i\pi}{L}(px+qy)}\,\d x\d y,\quad p,q=0,\pm1,\pm2,\dots.
\end{equation}
It is easy to see that the Fourier coefficients satisfy the following ODEs:
$$\daoshu{}{t}\widehat{u}_{pq}(t)=\lambda_{pq}\widehat{u}_{pq}(t),
\quad\lambda_{pq}=\frac{\pi^2(p^2+q^2)}{L^2}-\delta\bigg(\frac{\pi^2(p^2+q^2)}{L^2}\bigg)^2.$$
The exact solution is $$\widehat{u}_{pq}(t+\tau)=\e^{\lambda_{pq}\tau}\widehat{u}_{pq}(t),\quad p,q=0,\pm1,\pm2,\dots,$$
and then
\begin{equation}
\label{linear2D_solution_series}
u(x,y,t+\tau)=\sum_{p=-\infty}^\infty\sum_{q=-\infty}^\infty\widehat{u}_{pq}(t+\tau)\e^{\frac{\i\pi}{L}(px+qy)}.
\end{equation}
It is easy to see that the pseudo-spectral method purposed in \cite{ChKuQuTa15} is the discrete form of the procedure above.

In the theory of the spectral method \cite{ShTa06,ShTaWa11}, the FFT and the inverse FFT algorithm can be expressed as
\begin{equation}
\label{lieanr2D_discoef}
\widetilde{u}_{pq}(t)=\frac{1}{J^2c_pc_q}\sum_{j=1}^J\sum_{k=1}^Ju(x_j,y_k,t)\e^{-\frac{\i\pi}{L}(px_j+qy_k)},\quad p,q=-N,\dots,N,
\end{equation}
and
\begin{equation}
\label{lieanr2D_disseries}
u_{j,k}(t)=\sum_{p=-N}^N\sum_{q=-N}^N\widetilde{u}_{pq}(t)\e^{\frac{\i\pi}{L}(px_j+qy_k)},\quad j,k=1,2,\dots,J,
\end{equation}
where $c_p$ and $c_q$ are defined as
\begin{equation}
\label{coef_ck}
c_r=
\begin{dcases}
2, & |r|=N,\\
1, & |r|<N.
\end{dcases}
\end{equation}
The pseudo-spectral procedure can be expressed as
$$u(t+\tau)=\mathscr{F}_d^{-1}\{\e^{\lambda_{pq}\tau}\mathscr{F}_d[u(t)](p,q)\},$$
where $u(t)$ is the matrix with the elements $\{u_{j,k}(t):j,k=1,2,\dots,J\}$,
$\mathscr{F}_d$ and $\mathscr{F}_d^{-1}$ are the discrete Fourier transform and the inverse transform, respectively.
Using the Parseval's formula and the fact that $|\e^{\lambda_{pq}\tau}|\le\e^{\frac{\tau}{4\delta}}$ (for any $p,q$),
we obtain
\begin{equation}
\label{linear_stability}
\|u(t+\tau)\|\le\e^{\frac{\tau}{4\delta}}\|u(t)\|,
\end{equation}
where $\|\cdot\|$ represents the discrete $L^2$-norm, that is,
$$\|u\|=\sqrt{h^2\sum_{j=1}^J\sum_{k=1}^J(u_{j,k})^2}.$$
The inequality (\ref{linear_stability}) implies the stability of the pseudo-spectral procedure.

\section{Error analysis and accuracy tests}

Here we investigate the convergence rate of the fast explicit operator splitting method given above,
and then conduct some numerical accuracy tests to verify our results.

\subsection{Error estimate}

We denote by $\widetilde{u}(x,y,t)$ the splitting solution satisfying exactly the scheme (\ref{StrangOS_LNL}),
and write $U_{jk}^n:=u(x_j,y_k,t_n)$, $\widetilde{u}^n:=\widetilde{u}(\cdot,\cdot,t_n)$ and $\widetilde{U}_{jk}^n:=\widetilde{u}(x_j,y_k,t_n)$.
We denote by $\SNh$ and $\SLh$ the discrete approximations of the operators $\mathcal{S_N}$ and $\mathcal{S_L}$, respectively,
and by $u_{jk}^n$ the numerical approximation of $\widetilde{U}_{jk}^n$, satisfying
$$u^{n+1}=\SLh\Big(\frac{\tau}{2}\Big)\SNh(\tau)\SLh\Big(\frac{\tau}{2}\Big)u^n.$$
Defining a sample operator $I^h:\Lp^2(\Omega)\to\R^{J\times J}$ as $I^hu=(u(x_j,y_k))_{jk}$,
we have $\widetilde{U}^n=I^h\widetilde{u}^n$, where $\Lp^2(\Omega)=\{u\in L^2(\Omega)\,|\,u~\text{is $\Omega$-periodic}\}$.
For the simple notations, we omit the $\frac{\tau}{2}$ or $\tau$ following the symbols $\SN$, $\SL$, $\SNh$, or $\SLh$ below.

To estimate the error, we need some lemmas.
For the simplicity, we write $\SN u$ to mean $\SN(u)$ and $\SNh v$ to mean $\SNh(v)$, though the operators $\SN$ and $\SNh$ are actually nonlinear.
We restate the accuracy of $\SNh$ obtained in Section \ref{ss_nonlinear}.

\begin{lemma}
\label{lem_error_SNh}
Under the condition {\rm(\ref{nonlinear_stability})}, there exists a positive constant $C_1$, independent on $\tau$ and $h$, such that
$$\|I^h\SN u-\SNh I^hu\|\le C_1\tau(\tau^3+h^4),\quad\forall u\in\Hp^2(\Omega).$$
\end{lemma}

We can derive the stability of $\SNh$ by using the result of $\SN$.

\begin{lemma}
\label{lem_stability_SNh}
Given $m\in\N$.
Under the condition {\rm(\ref{nonlinear_stability})}, there exists a positive constant $C_2$, independent on $\tau$ and $h$, such that
$$\|\SNh v-\SNh w\|\le\|v-w\|+2C_1\tau(\tau^3+h^4)+C_2h^m,\quad\forall v,w\in\R^{J\times J}.$$
\end{lemma}

\begin{proof}
Let $\overline{v}$ be some function, belonging to $\Hp^m(\Omega)$, such that $I^h\overline{v}=v$,
for example, the two-dimensional trigonometric interpolation of $v$ in $\Omega$.
Similarly, let $\overline{w}\in\Hp^m(\Omega)$ such that $I^h\overline{w}=w$.
Using Lemma \ref{lem_error_SNh}, we obtain
\begin{align*}
\|\SNh v-\SNh w\| & \le\|\SNh v-I^h\SN\overline{v}\|+\|I^h\SN\overline{v}-I^h\SN\overline{w}\|+\|I^h\SN\overline{w}-\SNh w\|\\
& \le C_1\tau(\tau^3+h^4)+\|I^h(\SN\overline{v}-\SN\overline{w})\|+C_1\tau(\tau^3+h^4).
\end{align*}
Since the $L^2$-norm of an $\Omega$-periodic function on $\Omega$ can be approximated by the discrete $L^2$-norm with spectral accuracy \cite{Tre00},
using Proposition \ref{prop_stability_SN}, we have
\begin{align*}
\|I^h(\SN\overline{v}-\SN\overline{w})\| & \le\|\SN\overline{v}-\SN\overline{w}\|_{L^2(\Omega)}+Ch^m\\
& \le\|\overline{v}-\overline{w}\|_{L^2(\Omega)}+Ch^m\le\|v-w\|+C_2h^m.
\end{align*}
Therefore, we obtain
$$\|\SNh v-\SNh w\|\le\|v-w\|+2C_1\tau(\tau^3+h^4)+C_2h^m,$$
which completes the proof.\qquad
\end{proof}

\begin{remark}
In {\rm\cite{LiLi03}}, the authors have proved the regularity of the solutions to the MBE equation {\rm(\ref{MBE2D_model1})}
using the standard technique of Galerkin approximations.
It says that $u(t)\in\Hp^m(\Omega)$ for any $t>0$ if $u(0)\in\Hp^m(\Omega)$.
With the similar proof, we can obtain $\SN(t)u\in\Hp^m(\Omega)$ and $\SL(t)u\in\Hp^m(\Omega)$ for any $t>0$ provided $u\in\Hp^m(\Omega)$.
Here we omit the detailed proofs and just use the results directly above.
\end{remark}

The stability inequality (\ref{linear_stability}) can be rewritten in the following form.

\begin{lemma}
\label{lem_stability_SLh}
$\|\SLh v\|\le\e^{\frac{\tau}{4\delta}}\|v\|$, $\forall v\in\R^{J\times J}$.
\end{lemma}

The error estimate of the operator $\SLh$ defined in Section \ref{ss_linear} can be proved in the framework of spectral method.

\begin{lemma}
\label{lem_error_SLh}
Given $m\in\N$ and $m>1$.
There exists a positive constant $C_3$, independent on $\tau$ and $h$, such that
$$\|I^h\SL u-\SLh I^hu\|\le C_3|u|_m\e^{\frac{\tau}{4\delta}}h^{m},\quad\forall u\in\Hp^m(\Omega).$$
\end{lemma}

\begin{proof}
We use the notations $w(x,y,t)$, $\widehat{w}(x,y,t)$ and $\widetilde{w}(x,y,t)$ representing
\begin{align*}
w(x,y,t) & =\sum_{p=-\infty}^\infty\sum_{q=-\infty}^\infty\widehat{u}_{pq}(t)\e^{\frac{\i\pi}{L}(px+qy)},\quad
\overline{w}(x,y,t)=\sum_{p=-\infty}^\infty\sum_{q=-N}^N\widehat{u}_{pq}(t)\e^{\frac{\i\pi}{L}(px+qy)},\\
\widehat{w}(x,y,t) & =\sum_{p=-N}^N\sum_{q=-N}^N\widehat{u}_{pq}(t)\e^{\frac{\i\pi}{L}(px+qy)},\quad
\widetilde{w}(x,y,t)=\sum_{p=-N}^N\sum_{q=-N}^N\widetilde{u}_{pq}(t)\e^{\frac{\i\pi}{L}(px+qy)},
\end{align*}
where $\widehat{u}_{pq}(t)$ and $\widetilde{u}_{pq}(t)$ are given by (\ref{linear2D_coef}) and (\ref{lieanr2D_discoef}), respectively,
Assuming that $w(\cdot,\cdot,0)=\overline{w}(\cdot,\cdot,0)=\widehat{w}(\cdot,\cdot,0)=\widetilde{w}(\cdot,\cdot,0)=u$, we know that
$$\|I^h\SL u-\SLh I^hu\|^2=h^2\sum_{j=1}^{2N}\sum_{k=1}^{2N}|w(x_j,y_k,\tau)-\widetilde{w}(x_j,y_k,\tau)|^2\le4A_1+4A_2+2B,$$
where
\begin{align*}
A_1 & =h^2\sum_{j=1}^{2N}\sum_{k=1}^{2N}|w(x_j,y_k,\tau)-\overline{w}(x_j,y_k,\tau)|^2,\\
A_2 & =h^2\sum_{j=1}^{2N}\sum_{k=1}^{2N}|\overline{w}(x_j,y_k,\tau)-\widehat{w}(x_j,y_k,\tau)|^2,\\
B & =h^2\sum_{j=1}^{2N}\sum_{k=1}^{2N}|\widehat{w}(x_j,y_k,\tau)-\widetilde{w}(x_j,y_k,\tau)|^2.
\end{align*}

We first estimate the terms $A_1$ and $A_2$.
Since
\begin{align*}
A_1 & =h^2\sum_{j=1}^{2N}\sum_{k=1}^{2N}\bigg|\sum_{p=-\infty}^\infty\sum_{|q|>N}\widehat{u}_{pq}(\tau)\e^{\frac{\i\pi}{L}(px_j+qy_k)}\bigg|^2\\
& =4L^2\sum_{p=-\infty}^\infty\sum_{|q|>N}|\widehat{u}_{pq}(\tau)|^2\\
& \le4L^2\e^\frac{\tau}{2\delta}\sum_{p=-\infty}^\infty\sum_{|q|>N}|\widehat{u}_{pq}(0)|^2\\
& \le4L^2\e^\frac{\tau}{2\delta}\sum_{p=-\infty}^\infty\bigg(N^{-2m}\sum_{|q|>N}q^{2m}|\widehat{u}_{pq}(0)|^2\bigg)\\
& =4L^2\e^\frac{\tau}{2\delta}N^{-2m}\sum_{p=-\infty}^\infty\sum_{|q|>N}q^{2m}|\widehat{u}_{pq}(0)|^2,
\end{align*}
and, similarly,
\begin{align*}
A_2 & =h^2\sum_{j=1}^{2N}\sum_{k=1}^{2N}\bigg|\sum_{q=-N}^N\sum_{|p|>N}\widehat{u}_{pq}(\tau)\e^{\frac{\i\pi}{L}(px_j+qy_k)}\bigg|^2\\
& \le4L^2\e^\frac{\tau}{2\delta}N^{-2m}\sum_{q=-N}^N\sum_{|p|>N}p^{2m}|\widehat{u}_{pq}(0)|^2,
\end{align*}
we obtain
\begin{align*}
A_1+A_2 & \le4L^2\e^\frac{\tau}{2\delta}N^{-2m}
\bigg(\sum_{p=-\infty}^\infty\sum_{|q|>N}q^{2m}|\widehat{u}_{pq}(0)|^2+\sum_{q=-N}^N\sum_{|p|>N}p^{2m}|\widehat{u}_{pq}(0)|^2\bigg)\\
& \le4L^2\e^\frac{\tau}{2\delta}N^{-2m}\sum_{p=-\infty}^\infty\sum_{q=-\infty}^\infty(p^{2m}+q^{2m})|\widehat{u}_{pq}(0)|^2\\
& \le4L^2\e^\frac{\tau}{2\delta}N^{-2m}\sum_{p=-\infty}^\infty\sum_{q=-\infty}^\infty(p^2+q^2)^m|\widehat{u}_{pq}(0)|^2\\
& =4L^2\e^\frac{\tau}{2\delta}N^{-2m}|u|_m^2,
\end{align*}
where $|\cdot|_m$ represents the semi-norm of $\Hp^m(\Omega)$.

We next estimate the term $B$.
It is easy to obtain
\begin{align*}
B & =h^2\sum_{j=1}^{2N}\sum_{k=1}^{2N}\bigg|\sum_{p=-N}^N\sum_{q=-N}^N(\widehat{u}_{pq}(\tau)-\widetilde{u}_{pq}(\tau))\e^{\frac{\i\pi}{L}(px+qy)}\bigg|^2\\
& =4L^2\sum_{p=-N}^N\sum_{q=-N}^N|\widehat{u}_{pq}(\tau)-\widetilde{u}_{pq}(\tau)|^2\\
& \le4L^2\e^\frac{\tau}{2\delta}\sum_{p=-N}^N\sum_{q=-N}^N|\widehat{u}_{pq}(0)-\widetilde{u}_{pq}(0)|^2.
\end{align*}
Now we look for the upper bound of $B$ via the following fourth steps.

(i) Magnify the sum $$D:=\sum_{p=-N}^N\sum_{q=-N}^N|\widehat{u}_{pq}-\widetilde{u}_{pq}|^2.$$
A direct calculation leads to
\begin{align*}
D & =\sum_{p=-N}^N\bigg(\sum_{q=-N+1}^{N-1}|\widehat{u}_{pq}-\widetilde{u}_{pq}|^2+\frac{1}{4}\sum_{q=\pm N}|2\widehat{u}_{pq}-2\widetilde{u}_{pq}|^2\bigg)\\
& \le\sum_{p=-N}^N\bigg(\sum_{q=-N+1}^{N-1}|\widehat{u}_{pq}-\widetilde{u}_{pq}|^2
+\frac{1}{2}\sum_{q=\pm N}|\widehat{u}_{pq}-2\widetilde{u}_{pq}|^2+\frac{1}{2}\sum_{q=\pm N}|\widehat{u}_{pq}|^2\bigg)\\
& \le\sum_{p=-N}^N\bigg(\sum_{q=-N}^N|\widehat{u}_{pq}-c_q\widetilde{u}_{pq}|^2+\frac{1}{2}\sum_{q=\pm N}|\widehat{u}_{pq}|^2\bigg)\\
& =\sum_{q=-N}^N\bigg(\sum_{p=-N}^N|\widehat{u}_{pq}-c_q\widetilde{u}_{pq}|^2\bigg)+\frac{1}{2}\sum_{p=-N}^N\sum_{q=\pm N}|\widehat{u}_{pq}|^2\\
& \le\sum_{q=-N}^N\bigg(\sum_{p=-N}^N|\widehat{u}_{pq}-c_pc_q\widetilde{u}_{pq}|^2+\frac{1}{2}\sum_{p=\pm N}|\widehat{u}_{pq}|^2\bigg)
+\frac{1}{2}\sum_{p=-N}^N\sum_{q=\pm N}|\widehat{u}_{pq}|^2\\
& =\sum_{p=-N}^N\sum_{q=-N}^N|\widehat{u}_{pq}-c_pc_q\widetilde{u}_{pq}|^2
+\frac{1}{2}\sum_{q=-N}^N\sum_{p=\pm N}|\widehat{u}_{pq}|^2+\frac{1}{2}\sum_{p=-N}^N\sum_{q=\pm N}|\widehat{u}_{pq}|^2\\
& =:D_1+D_2+D_3,
\end{align*}
where $c_p$ and $c_q$ are defined as (\ref{coef_ck}).

(ii) Estimate the term $D_2+D_3$.
Since
\begin{align*}
D_2 & =\frac{1}{2}\sum_{q=-N}^N\sum_{p=\pm N}|\widehat{u}_{pq}|^2
\le\frac{1}{2}\sum_{q=-N}^N\bigg(N^{-2m}\sum_{|p|\ge N}p^{2m}|\widehat{u}_{pq}|^2\bigg)\\
& =\frac{1}{2}N^{-2m}\sum_{q=-N}^N\sum_{|p|\ge N}p^{2m}|\widehat{u}_{pq}|^2
\le\frac{1}{2}N^{-2m}\sum_{q=-\infty}^\infty\sum_{p=-\infty}^\infty p^{2m}|\widehat{u}_{pq}|^2,
\end{align*}
and, similarly,
$$D_3\le\frac{1}{2}N^{-2m}\sum_{p=-\infty}^\infty\sum_{q=-\infty}^\infty q^{2m}|\widehat{u}_{pq}|^2,$$
we obtain
$$D_2+D_3\le\frac{1}{2}N^{-2m}\sum_{q=-\infty}^\infty\sum_{p=-\infty}^\infty(p^2+q^2)^m|\widehat{u}_{pq}|^2=\frac{1}{2}N^{-2m}|u|_m^2.$$

(iii) To estimate the term $D_1$, we first prove that
\begin{equation}
\label{aliaingformula}
c_pc_q\widetilde{u}_{pq}=\widehat{u}_{pq}+\sum_{r^2+s^2\not=0}\widehat{u}_{p+2rN,q+2sN}.
\end{equation}
In fact, substituting (\ref{linear2D_series}) into (\ref{lieanr2D_discoef}), we have
\begin{align*}
c_pc_q\widetilde{u}_{pq}
& =\frac{1}{4N^2}\sum_{j=1}^{2N}\sum_{k=1}^{2N}
\bigg(\sum_{r=-\infty}^\infty\sum_{s=-\infty}^\infty\widehat{u}_{rs}\e^{\frac{\i\pi}{L}(rx_j+sy_k)}\bigg)\e^{-\frac{\i\pi}{L}(px_j+qy_k)}\\
& =\frac{1}{4N^2}\sum_{j=1}^{2N}\sum_{k=1}^{2N}\sum_{r=-\infty}^\infty\sum_{s=-\infty}^\infty\widehat{u}_{rs}\e^{\frac{\i\pi}{L}((r-p)x_j+(s-q)y_k)}\\
& =\frac{1}{4N^2}\sum_{r=-\infty}^\infty\sum_{s=-\infty}^\infty\widehat{u}_{rs}
\sum_{j=1}^{2N}\e^{\frac{\i\pi}{L}(r-p)x_j}\sum_{k=1}^{2N}\e^{\frac{\i\pi}{L}(s-q)y_k}\\
& =\sum_{r=-\infty}^\infty\sum_{s=-\infty}^\infty\widehat{u}_{p+2rN,q+2sN}\\
& =\widehat{u}_{pq}+\sum_{r^2+s^2\not=0}\widehat{u}_{p+2rN,q+2sN},
\end{align*}
since
\[
\sum_{j=1}^{2N}\e^{\frac{\i\pi}{L}(r-p)x_j}=
\begin{dcases}
2N, & r-p=2lN,\\
0, & r-p\not=2lN,
\end{dcases}\quad\text{here $l$ is an integer.}
\]

(iv) Estimate the term $D_1$.
Using the formula (\ref{aliaingformula}) and the Cauchy-Schwarz inequality, we have
\begin{align*}
D_1 & =\sum_{p=-N}^N\sum_{q=-N}^N\bigg|\sum_{r^2+s^2\not=0}\widehat{u}_{p+2rN,q+2sN}\bigg|^2\\
& \le \sum_{p=-N}^N\sum_{q=-N}^N\bigg\{\bigg(\sum_{r^2+s^2\not=0}\big[(p+2rN)^2+(q+2sN)^2\big]^{-m}\bigg)\cdot\\
& \qquad\qquad\qquad\qquad\bigg(\sum_{r^2+s^2\not=0}\big[(p+2rN)^2+(q+2sN)^2\big]^m|\widehat{u}_{p+2rN,q+2sN}|^2\bigg)\bigg\}\\
& \le \max_{|p|,|q|\le N}\bigg(\sum_{r^2+s^2\not=0}\big[(p+2rN)^2+(q+2sN)^2\big]^{-m}\bigg)\cdot\\
& \qquad\qquad\bigg(\sum_{p=-N}^N\sum_{q=-N}^N\sum_{r^2+s^2\not=0}\big[(p+2rN)^2+(q+2sN)^2\big]^m|\widehat{u}_{p+2rN,q+2sN}|^2\bigg)\\
& \le \bigg(\sum_{r^2+s^2\not=0}\big[(2rN-N)^2+(2sN-N)^2\big]^{-m}\bigg)\cdot2|u|_m^2\\
& =2N^{-2m}|u|_m^2\sum_{r^2+s^2\not=0}\frac{1}{\big[(2r-1)^2+(2s-1)^2\big]^m}.
\end{align*}
The series
\begin{align*}
\sum_{r^2+s^2\not=0}\frac{1}{\big[(2r-1)^2+(2s-1)^2\big]^m}
& \le\sum_{r^2+s^2\not=0}\frac{1}{(2r-1)^{2m}+(2s-1)^{2m}}\\
& \le\frac{1}{2}\sum_{r^2+s^2\not=0}\frac{1}{|2r-1|^{m}|2s-1|^{m}}<\infty,\quad\text{if $m>1$},
\end{align*}
so we obtain $D_1\le2SN^{-2m}|u|_m^2$, where $S$ is the sum of the series above.

As a result of (i)--(iv), we obtain
$$B\le2(1+4S)L^2\e^\frac{\tau}{2\delta}N^{-2m}|u|_m^2.$$
So we obtain
$$\|I^h\SL u-\SLh I^hu\|^2\le4(5+4S)L^2|u|_m^2\e^\frac{\tau}{2\delta}N^{-2m},$$
which leads to the expected result.\qquad
\end{proof}

Now we write the discrete $L^2$-error estimate as the following theorem.

\begin{theorem}
\label{thm_error2d}
Assume that $u_0\in\Hp^m(\Omega)$ with $m>1$ and the condition {\rm(\ref{nonlinear_stability})} holds.
If we set $u^0=\widetilde{U}^0=I^hu_0$, then the discrete $L^2$-error at $T=n\tau$ is
\begin{equation}
\label{error_estimate}
\|U^{n}-u^{n}\|\le C\Big(\tau^2+h^4+\frac{h^m}{\tau}\Big).
\end{equation}
Furthermore, if $m\ge6$ and $\tau\sim h^2$, then
\begin{equation}
\label{error_estimate_optimal}
\|U^{n}-u^{n}\|\le C(\tau^2+h^4).
\end{equation}
\end{theorem}

\begin{proof}
Assume that $u^{n-1}$, the numerical solution at $t_{n-1}$-level, is given,
then the discrete $L^2$-error at $t_{n}$-level should be
\begin{equation}
\label{estimate1}
\|U^{n}-u^{n}\|\le\|U^{n}-\widetilde{U}^{n}\|+\|\widetilde{U}^{n}-u^{n}\|.
\end{equation}
The Strang splitting scheme (\ref{StrangOS_LNL}) is second-order \cite{Str68}, which means that
$$\|U^{n}-\widetilde{U}^{n}\|\le C_0\tau^2.$$
The second term in the RHS of (\ref{estimate1}) can be bounded as follows:
\begin{align}
\|\widetilde{U}^{n}-u^{n}\| & =\|I^h\SL\SN\SL\widetilde{u}^{n-1}-\SLh\SNh\SLh u^{n-1}\|\nonumber\\
& \le\|I^h\SL\SN\SL\widetilde{u}^{n-1}-\SLh I^h\SN\SL\widetilde{u}^{n-1}\|+\|\SLh I^h\SN\SL\widetilde{u}^{n-1}-\SLh\SNh\SLh u^{n-1}\|\nonumber\\
& =\|(I^h\SL-\SLh I^h)\SN\SL\widetilde{u}^{n-1}\|+\|\SLh(I^h\SN\SL\widetilde{u}^{n-1}-\SNh\SLh u^{n-1})\|\nonumber\\
& \le C_3|\SN\SL\widetilde{u}^{n-1}|_m\e^{\frac{\tau}{4\delta}}h^{m}
+\e^{\frac{\tau}{4\delta}}\|I^h\SN\SL\widetilde{u}^{n-1}-\SNh\SLh u^{n-1}\|,\label{estimate2}
\end{align}
where the last inequality is the consequences of Lemmas \ref{lem_stability_SLh} and \ref{lem_error_SLh}.
Besides,
\begin{align}
\|I^h\SN\SL\widetilde{u}^{n-1}-\SNh\SLh u^{n-1}\|
& \le\|I^h\SN\SL\widetilde{u}^{n-1}-\SNh I^h\SL\widetilde{u}^{n-1}\|+\|\SNh I^h\SL\widetilde{u}^{n-1}-\SNh\SLh u^{n-1}\|\nonumber\\
& =\|(I^h\SN-\SNh I^h)\SL\widetilde{u}^{n-1}\|+\|\SNh(I^h\SL\widetilde{u}^{n-1})-\SNh(\SLh u^{n-1})\|\nonumber\\
& \le3C_1\tau(\tau^3+h^4)+C_2h^m+\|I^h\SL\widetilde{u}^{n-1}-\SLh u^{n-1}\|,\label{estimate3}
\end{align}
where the last inequality is the consequences of Lemmas \ref{lem_error_SNh} and \ref{lem_stability_SNh}.
Furthermore,
\begin{align}
\|I^h\SL\widetilde{u}^{n-1}-\SLh u^{n-1}\|
& \le\|I^h\SL\widetilde{u}^{n-1}-\SLh I^h\widetilde{u}^{n-1}\|+\|\SLh I^h\widetilde{u}^{n-1}-\SLh u^{n-1}\|\nonumber\\
& =\|(I^h\SL-\SLh I^h)\widetilde{u}^{n-1}\|+\|\SLh(\widetilde{U}^{n-1}-u^{n-1})\|\nonumber\\
& \le C_3|\widetilde{u}^{n-1}|_m\e^{\frac{\tau}{4\delta}}h^{m}+\e^{\frac{\tau}{4\delta}}\|\widetilde{U}^{n-1}-u^{n-1}\|,\label{estimate4}
\end{align}
where we use the fact $\widetilde{U}^{n-1}=I^h\widetilde{u}^{n-1}$.
Combining (\ref{estimate2})--(\ref{estimate4}) with (\ref{estimate1}),
we obtain
$$\|\widetilde{U}^{n}-u^{n}\|\le\e^{\frac{\tau}{2\delta}}\|\widetilde{U}^{n-1}-u^{n-1}\|
+3C_1\e^{\frac{\tau}{4\delta}}\tau(\tau^3+h^4)+\big(C_2+C_3C_T(1+\e^{\frac{\tau}{4\delta}})\big)\e^{\frac{\tau}{4\delta}}h^{m},$$
where $C_T=\max\{|\SN\SL\widetilde{u}^{k}|_m,|\widetilde{u}^{k}|_m:0\le k\le n\}$.
Setting $F^n=\|\widetilde{U}^n-u^n\|$ and
$$G=3C_1\e^{\frac{\tau}{4\delta}}(\tau^3+h^4)+\big(C_2+C_3C_T(1+\e^{\frac{\tau}{4\delta}})\big)\e^{\frac{\tau}{4\delta}}\frac{h^{m}}{\tau},$$
we have$$F^{n}\le\e^{\frac{\tau}{2\delta}}F^{n-1}+\tau G,\quad n=1,2,\dots.$$
Using the Gronwall's lemma and the fact $\e^x-1\ge x$ ($x>0$), we obtain
$$F^n\le\e^{\frac{T}{2\delta}}F^{0}+\frac{\tau(\e^{\frac{T}{2\delta}}-1)}{\e^{\frac{\tau}{2\delta}}-1}G
\le\e^{\frac{T}{2\delta}}F^{0}+2\delta\e^{\frac{T}{2\delta}}G.$$
Since $F^0=\|\widetilde{U}^0-u^0\|=0$, we obtain
$$\|\widetilde{U}^n-u^n\|\le2\delta\e^{\frac{T}{2\delta}}\Big(3C_1\e^{\frac{\tau}{4\delta}}(\tau^3+h^4)
+\big(C_2+C_3C_T(1+\e^{\frac{\tau}{4\delta}})\big)\e^{\frac{\tau}{4\delta}}\frac{h^{m}}{\tau}\Big),$$
and thus
$$\|U^n-u^n\|\le C_0\tau^2+2\delta\e^{\frac{T}{2\delta}}\Big(3C_1\e^{\frac{\tau}{4\delta}}(\tau^3+h^4)
+\big(C_2+C_3C_T(1+\e^{\frac{\tau}{4\delta}})\big)\e^{\frac{\tau}{4\delta}}\frac{h^{m}}{\tau}\Big).$$
If $\tau\le\min\{4\delta\ln2,1\}$, then $\e^{\frac{\tau}{4\delta}}\le 2$, so we obtain
$$\|U^n-u^n\|\le(C_0+12\delta\e^{\frac{T}{2\delta}}C_1)\tau^2+12\delta\e^{\frac{T}{2\delta}}C_1h^4+4\delta\e^{\frac{T}{2\delta}}(C_2+3C_3C_T)\frac{h^{m}}{\tau},$$
which implies the estimate (\ref{error_estimate}).

Furthermore, we set the step $\tau\sim h^2$ to obtain
$$\frac{h^{m}}{\tau}\sim h^{m-2}\sim\tau^{\frac{1}{2}m-1}.$$
As long as $m\ge6$ holds, we obtain the error estimate (\ref{error_estimate_optimal}).\qquad
\end{proof}

\subsection{Accuracy tests}

Now we carry out the accuracy tests on the equation (\ref{MBE2D_model1}) with $\delta=0.1$, $T=1$, $\Omega=(0,2\pi)\times(0,2\pi)$, and
$$u_0(x,y)=0.1(\sin3x\sin2y+\sin5x\sin5y),$$
which is a classical example studied either theoretically or numerically \cite{LiLi03,QiSuZh12,QiZhTa11,XuTa06}.
We take the numerical solution obtained with $\tau=5\times10^{-5}$ and $J=2048$ as the ``exact'' solution.
The tests are conducted with different spatial scales, and the time step is set to be $\tau=C_0h^2$,
where the constant $C_0$ is chosen to render $\tau=0.005$ when $J=128$.
\tablename~\ref{tab_accuracy2d} shows the discrete $L^2$-errors implying the accuracy nearly $\mathcal{O}(\tau^2+h^4)$,
which is consistent with Theorem \ref{thm_error2d}.

\begin{table}[!htp]
\centering
\caption{The discrete $L^2$-errors with different spatial scales and time steps.}\label{tab_accuracy2d}
\begin{tabular}{|c|c|c|c|c|}
\hline
$J$ & $\tau$ & $\|e(J)\|$ & $\dfrac{\|e(J/2)\|}{\|e(J)\|}$ & $\log_2\dfrac{\|e(J/2)\|}{\|e(J)\|}$ \\
\hline
$128$ & $5\times10^{-3}$ & $1.0278\times10^{-5}$ & $*$ & $*$ \\
\hline
$256$ & $1.25\times10^{-3}$ & $9.5361\times10^{-7}$ & $10.7779$ & $3.4300$ \\
\hline
$512$ & $3.125\times10^{-4}$ & $6.5869\times10^{-8}$ & $14.4774$ & $3.8557$ \\
\hline
$1024$ & $7.8125\times10^{-5}$ & $2.4026\times10^{-9}$ & $27.4156$ & $4.7769$ \\
\hline
\end{tabular}
\end{table}

\section{Numerical experiments}

\quad

\begin{example}
\label{numerical_eg1}
We consider the one-dimensional MBE model
\begin{align*}
& u_t=(u_x^3)_x-u_{xx}-\delta u_{xxxx},\qquad\qquad\qquad\quad (x,t)\in(0,12)\times(0,T],\\
& u(\cdot,t)~\text{is $12$-periodic},\qquad\qquad\qquad\qquad\quad\,\ \ t\in[0,T],\\
& u(x,0)=0.1\Big(\sin\frac{\pi x}{2}+\sin\frac{2\pi x}{3}+\sin\pi x\Big),\ \ x\in[0,12].
\end{align*}
\end{example}

The evolution of this initial-boundary-value problem is studied theoretically via the perturbation analysis \cite{LiLi03}
to observe the morphological instability due to the nonlinear interaction.
It is also a classical example for the numerical experiments in the case $\delta=1$.
Here we will present the results obtained by the operator splitting method given in Section 2.
Fig.~\ref{fig_d0} shows the results of the case $\delta=1$ with $J=128$ and $\tau=0.1$,
which is consistent with the existing work \cite{LiLi03}.

\begin{figure}[h]
\centering
\subfigure[height: $t=0$]{\includegraphics[scale=0.65]{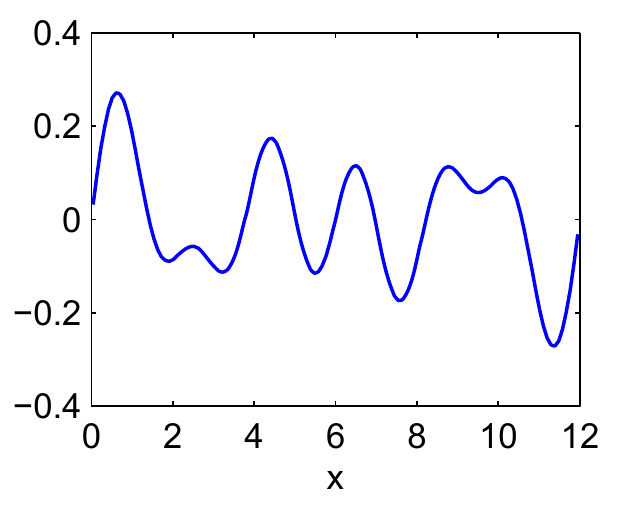}}
\subfigure[height: $t=0.5$]{\includegraphics[scale=0.65]{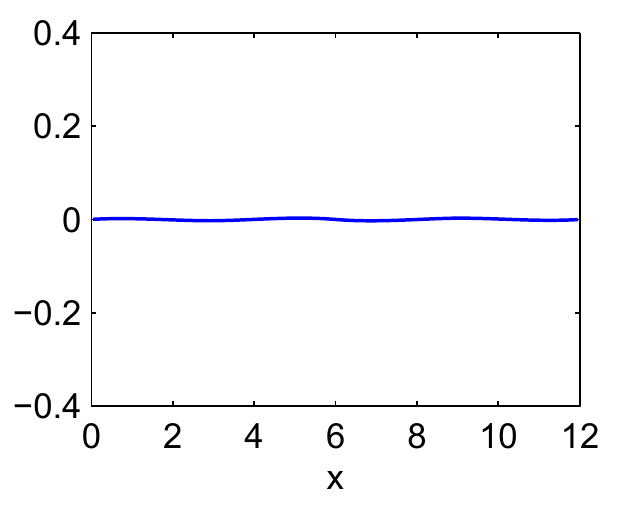}}
\subfigure[height: $t=15$]{\includegraphics[scale=0.65]{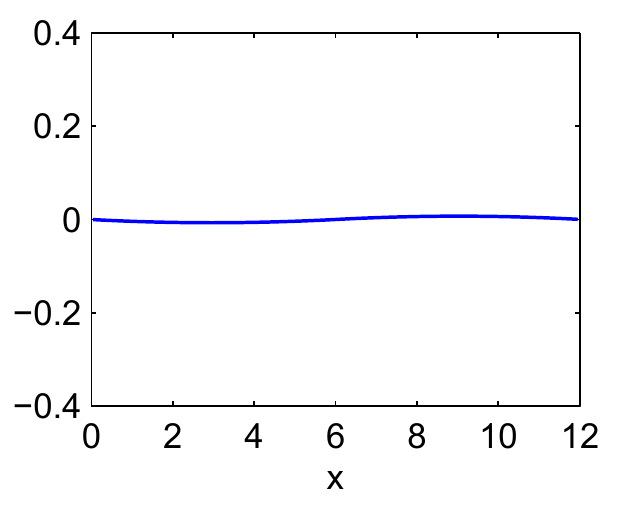}}
\subfigure[height: $t=20$]{\includegraphics[scale=0.65]{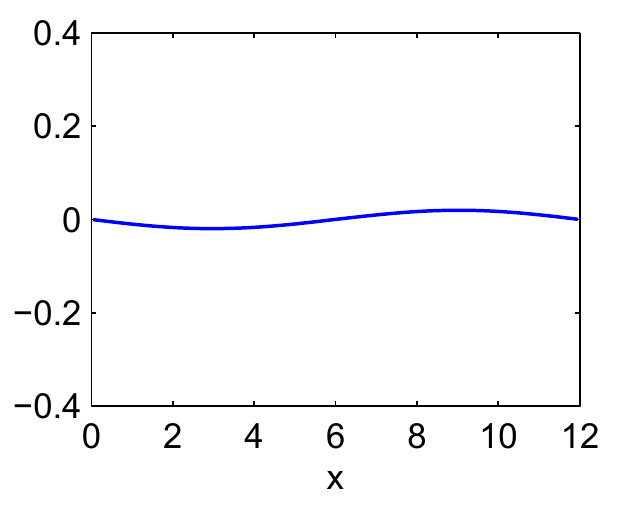}}
\subfigure[height: $t=30$]{\includegraphics[scale=0.65]{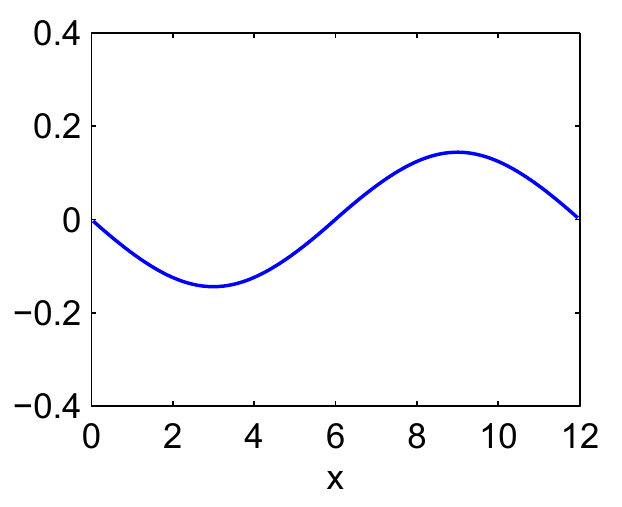}}
\subfigure[height: $t=60$]{\includegraphics[scale=0.65]{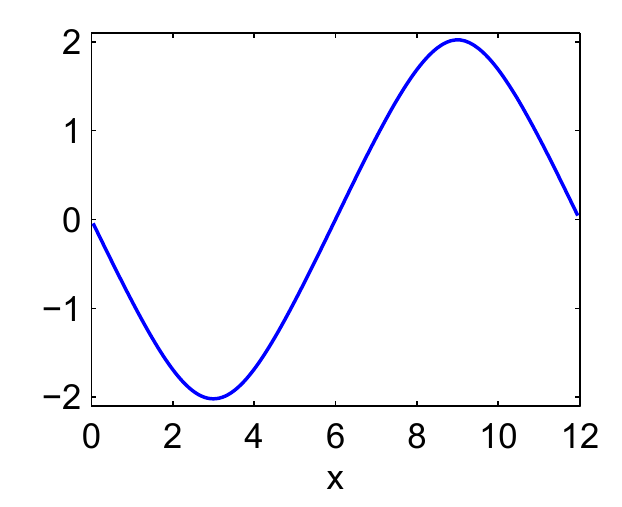}}
\subfigure[height: $t=100$]{\includegraphics[scale=0.65]{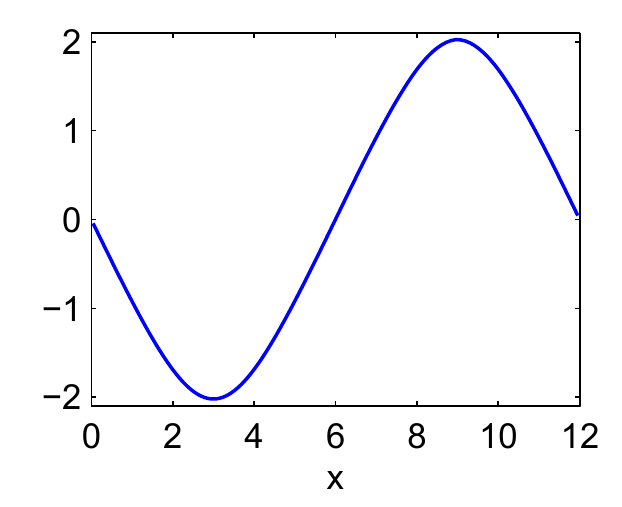}}
\subfigure[gradient: $t=100$]{\includegraphics[scale=0.65]{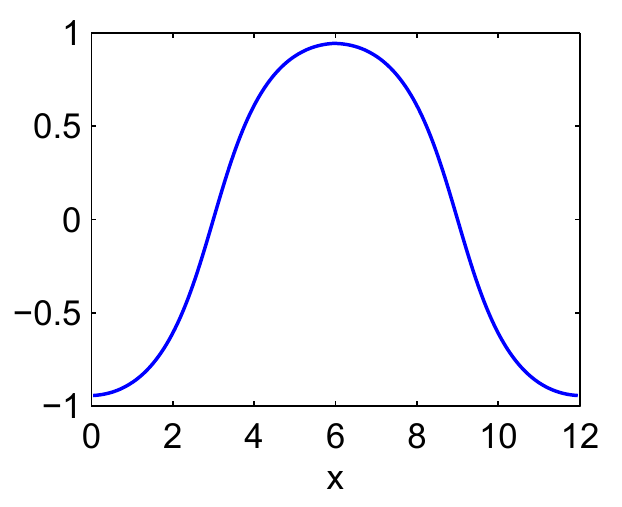}}
\subfigure[energy: $0\le t\le 100$]{\includegraphics[scale=0.65]{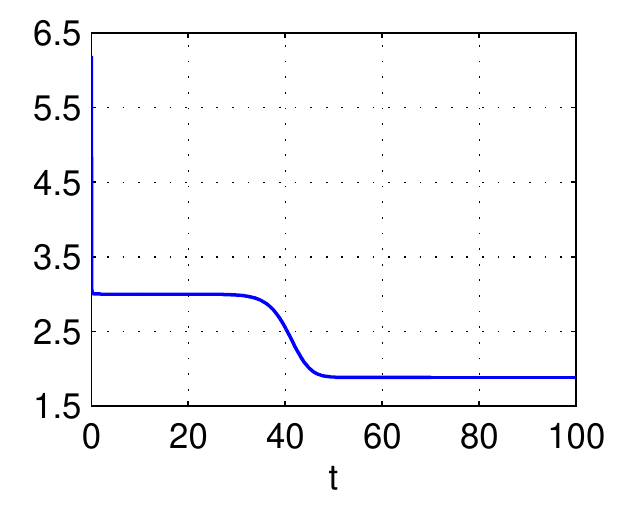}}
\caption{Example \ref{numerical_eg1}: The results of the case $\delta=1$.}\label{fig_d0}
\end{figure}

\begin{figure}[h]
\centering
\subfigure[height: $t=200$]{\includegraphics[scale=0.65]{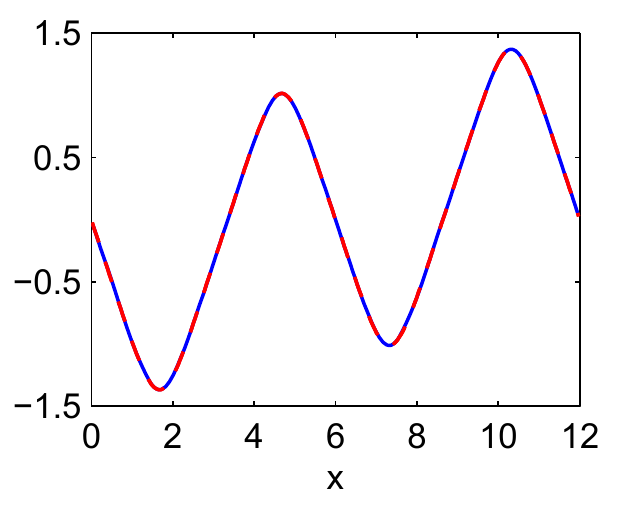}}
\subfigure[gradient: $t=200$]{\includegraphics[scale=0.65]{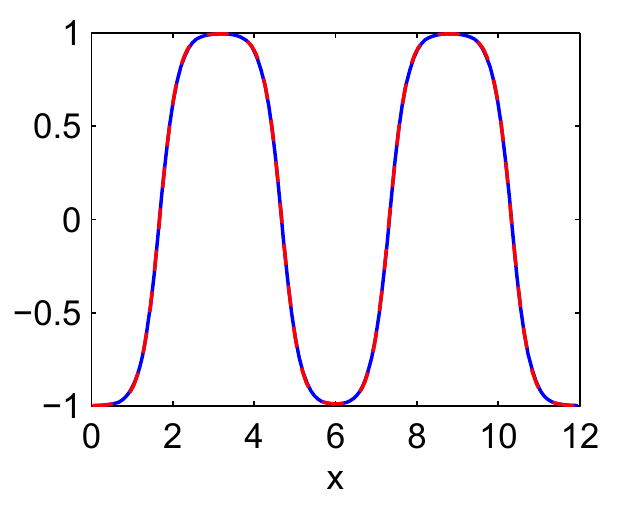}}
\subfigure[energy: $0\le t\le 200$]{\includegraphics[scale=0.65]{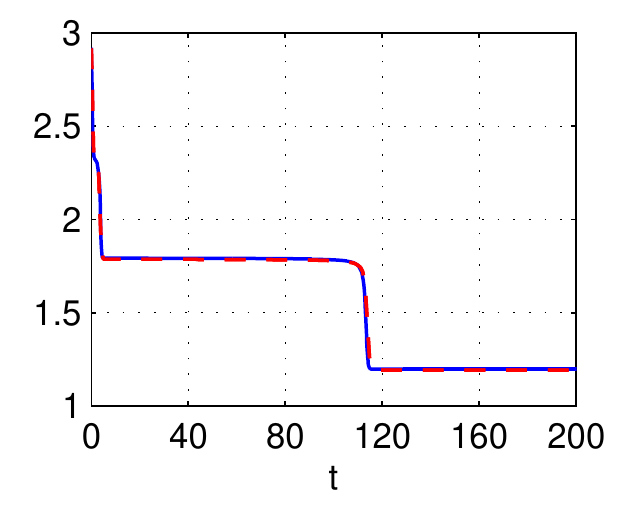}}
\caption{Example \ref{numerical_eg1}: The results of the case $\delta=0.1$
obtained with $(J,\tau)=(128,0.01)$ (solid line) and $(J,\tau)=(256,0.005)$ (dash line).}\label{fig_d1}
\end{figure}

\begin{figure}[h]
\centering
\subfigure[height: $t=500$]{\includegraphics[scale=0.65]{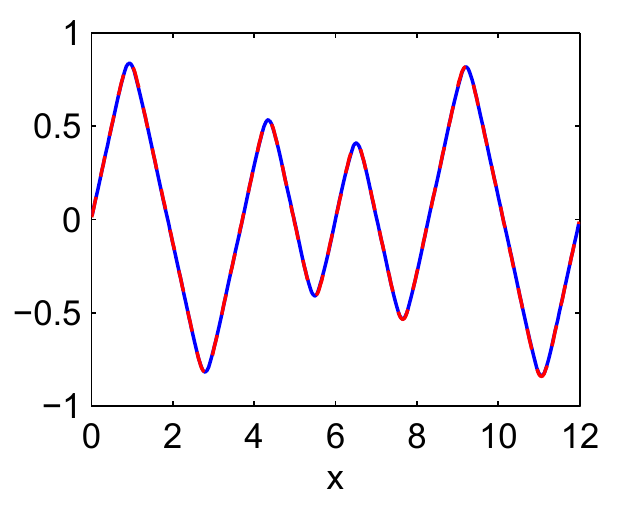}}
\subfigure[gradient: $t=500$]{\includegraphics[scale=0.65]{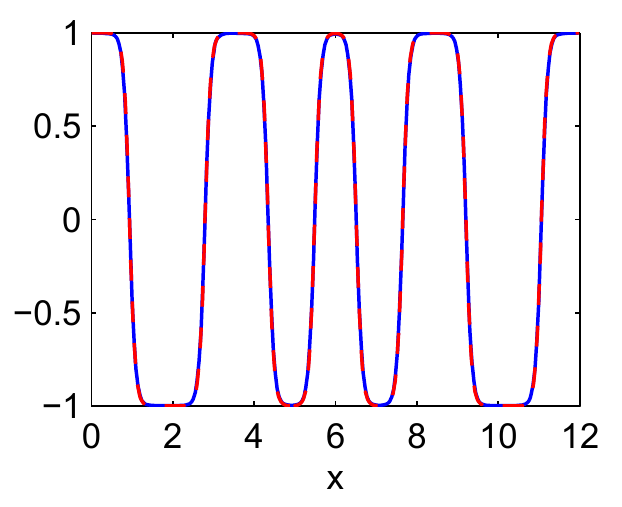}}
\subfigure[energy: $0\le t\le 500$]{\includegraphics[scale=0.65]{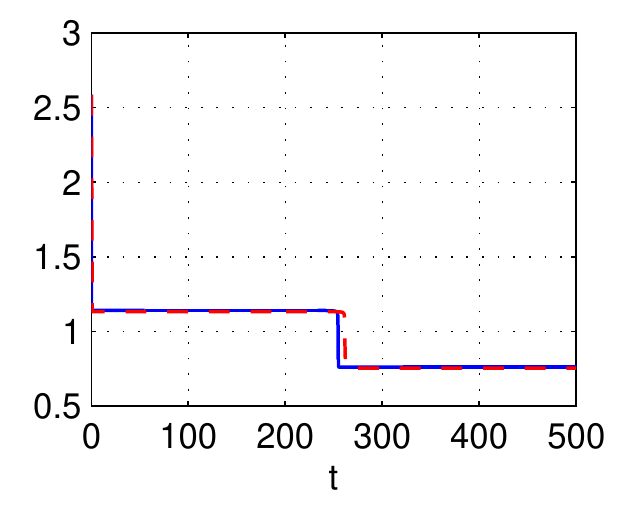}}
\caption{Example \ref{numerical_eg1}: The results of the case $\delta=0.01$
obtained with $(J,\tau)=(256,0.001)$ (solid line) and $(J,\tau)=(512,0.0005)$ (dash line).}\label{fig_d2}
\end{figure}

\begin{figure}[!htp]
\centering
\subfigure[height: $t=1000$]{\includegraphics[scale=0.65]{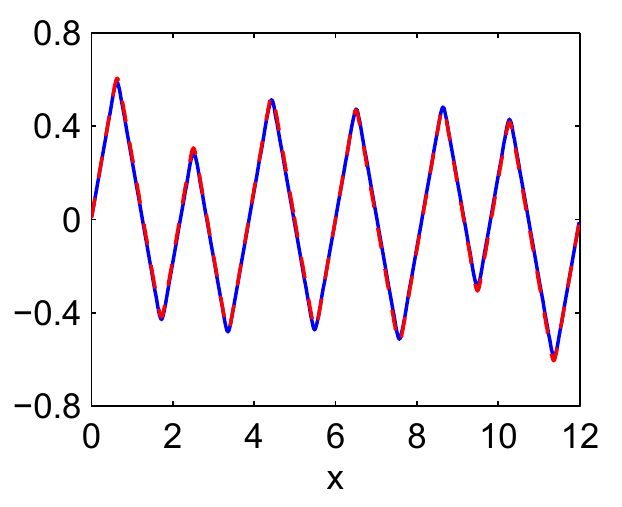}}
\subfigure[gradient: $t=1000$]{\includegraphics[scale=0.65]{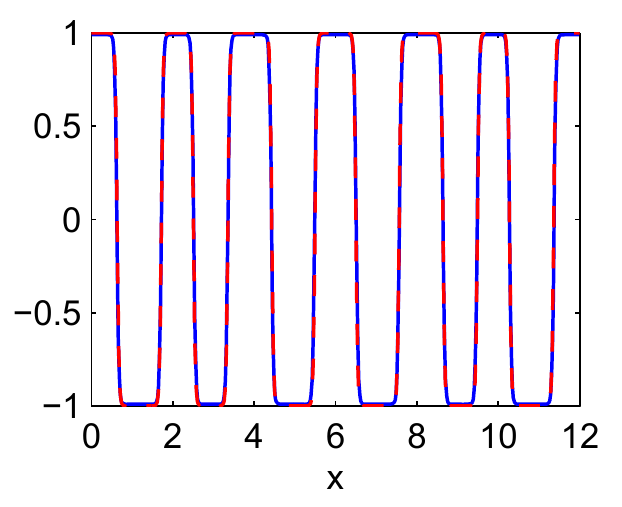}}
\subfigure[energy: $0\le t\le 1000$]{\includegraphics[scale=0.65]{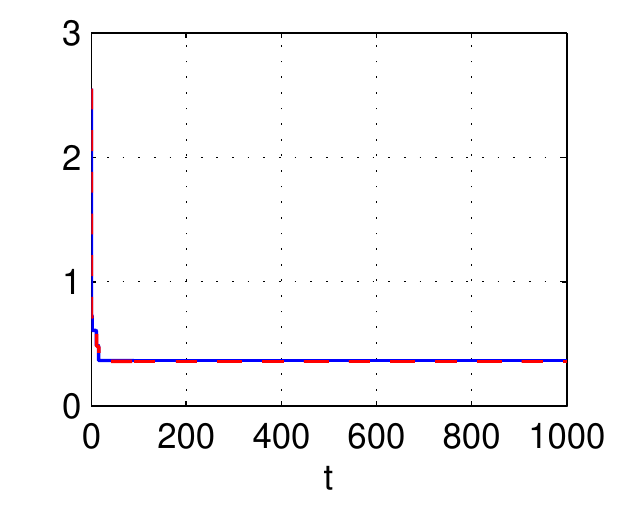}}
\caption{Example \ref{numerical_eg1}: The results of the case $\delta=0.001$
obtained with $(J,\tau)=(512,0.0001)$ (solid line) and $(J,\tau)=(1024,0.00005)$ (dash line).}\label{fig_d3}
\end{figure}

Besides, we also present some results from reducing $\delta$ to $0.1$, $0.01$, and $0.001$, respectively.
The results are summarized in Figs.~\ref{fig_d1}--\ref{fig_d3}.
The first plot in each figure presents the height $u(x,t)$ at some time $t$,
the second one shows the corresponding gradient $u_x(x,t)$,
and the third one plots the evolution of the energy $E(u(\cdot,t))$.

Fig.~\ref{fig_d1} presents the results of the case $\delta=0.1$ with $(J,\tau)=(128,0.01)$ and $(J,\tau)=(256,0.005)$.
We find that both solutions have few differences between each other, so we are convinced that the results presented here are credible.
The energy decreases hardly after $t=200$ so that we view the solution at $t=200$ as the steady state.
It is observed from the left and middle graphs that
there are two complete waves in the steady state whose gradients do not exceed the range between $-1$ and $1$.

Fig.~\ref{fig_d2} gives the results of the case $\delta=0.01$ with $(J,\tau)=(256,0.001)$ and $(J,\tau)=(512,0.0005)$.
Fig.~\ref{fig_d3} gives the results of the case $\delta=0.001$ with $(J,\tau)=(512,0.0001)$ and $(J,\tau)=(1024,0.00005)$.
Likewise, we can trust these results.
The solutions at the steady states present more waves in the considered domain than those above,
and the gradients still locate in the interval $[-1,1]$.

From the gradient graphs of Figs.~\ref{fig_d0}--\ref{fig_d3},
we find that the smaller $\delta$ is, the more points on the gradient curves locate at the horizon lines $y=1$ or $y=-1$.
This is a consequence of the competitions between the Ehrlich-Schwoebel effect and the dissipation mechanism of the energy functional
$$E(u)=\int_0^{2L}\Big(\frac{1}{4}(u_x^2-1)^2+\frac{\delta}{2}u_{xx}^2\Big)\,\d x.$$
It is seen that the Ehrlich-Schwoebel effect selects the slope $|u_x|=1$ while the dissipation term weakens the selection.
To reduce $\delta$ means to weaken the dissipation effect, or equivalently, to strengthen the slope selection.
And thus, the gradient interfaces connecting $-1$ to $1$ turn steep and the solution curves turn sharp there.

\begin{example}[Coarsening dynamics]
\label{numerical_eg2}
We simulate the two-dimensional MBE model {\rm(\ref{MBE2D_model1})} with $\delta=0.1$
on the domain $\Omega=(0,100)\times(0,100)$ with $512\times512$ grid.
We set the initial data as a stochastic state by a random number varying from $-0.001$ to $0.001$ on each grid point.
The time step is set to be $\tau=0.01$.
\end{example}

This example is aimed to verify the power laws for the energy evolution and the height growth.

Fig.~\ref{fig_energy} presents the contour lines of the free energy
$$F_{\text{free}}:=\frac{1}{4}(|\nabla u|^2-1)^2+\frac{\delta}{2}|\Delta u|^2$$
at $t=100$, $2000$, and $30000$.

\begin{figure}[!htp]
\centering
\subfigure[$t=100$]{\includegraphics[scale=0.45]{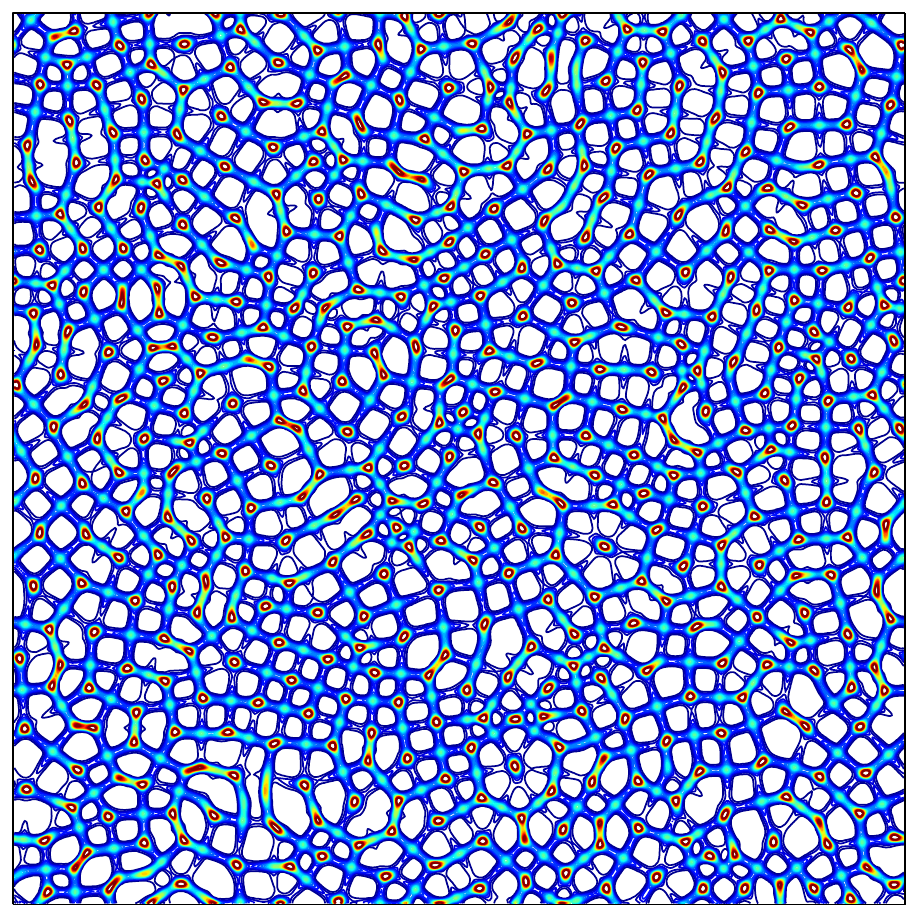}}
\subfigure[$t=2000$]{\includegraphics[scale=0.45]{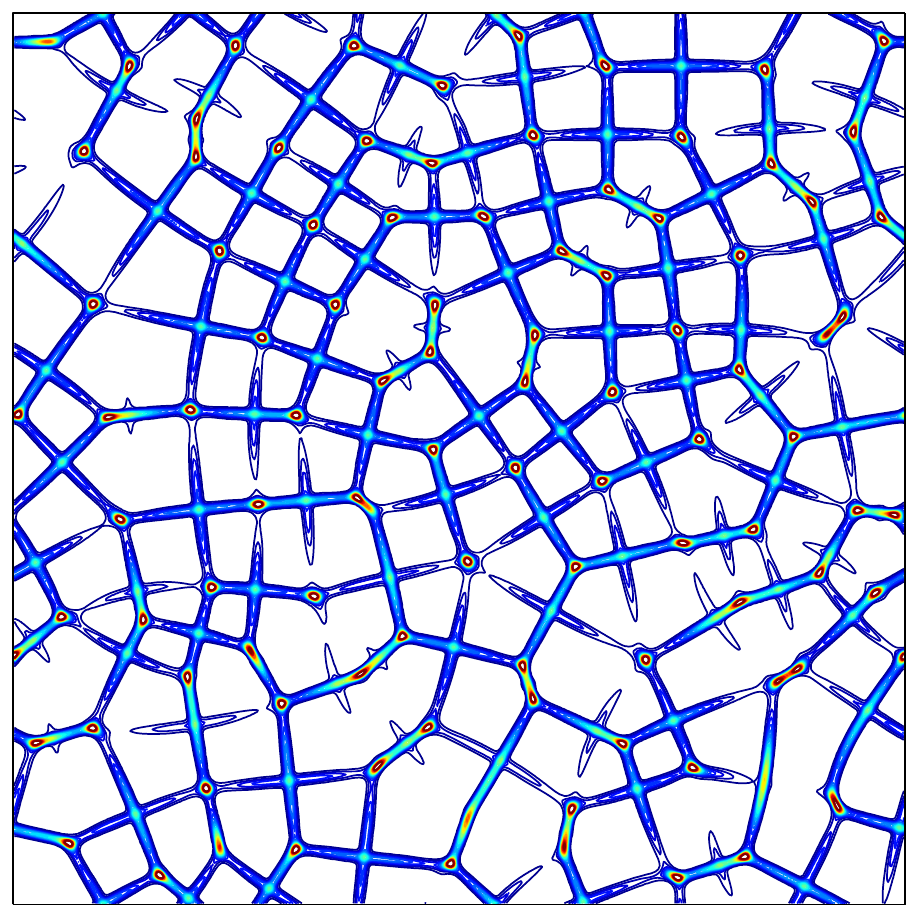}}
\subfigure[$t=30000$]{\includegraphics[scale=0.45]{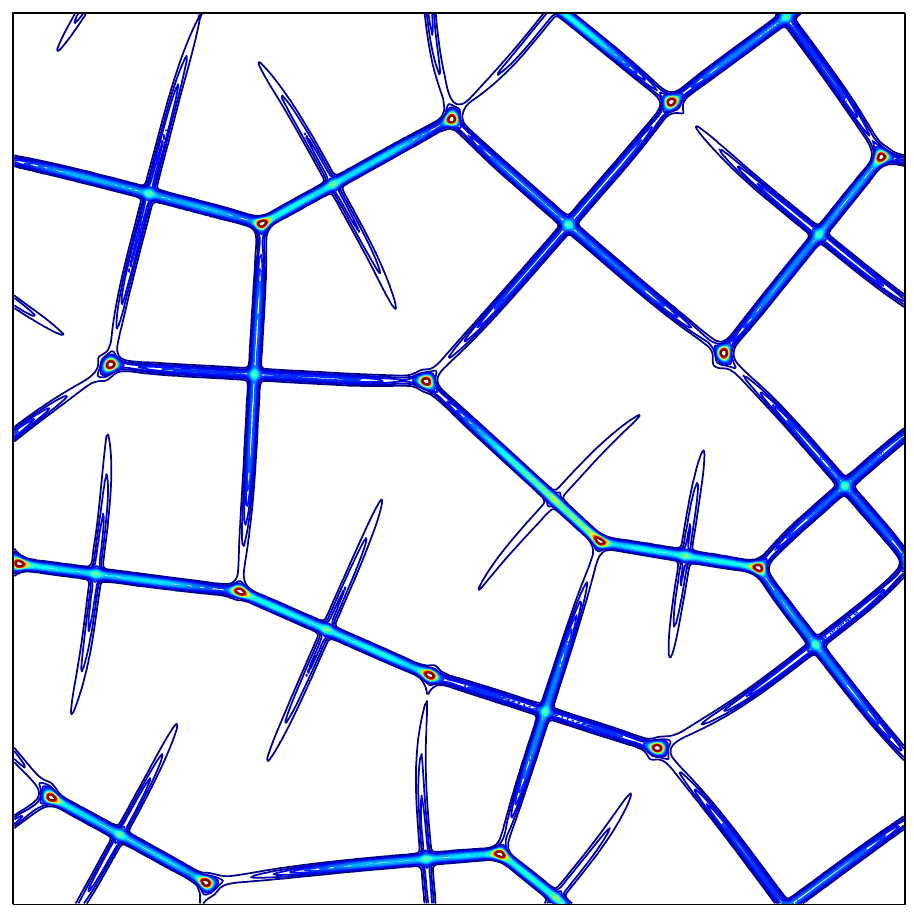}}
\caption{Example \ref{numerical_eg2}: Contour lines of the free energy $F_{\text{free}}$.}\label{fig_energy}
\end{figure}

Fig.~\ref{fig_evolution} shows the evolution of the energy and the interface height.
The (a) presents the power law for the evolution of the energy.
The energy curve is plotted in log-to-log scale and nearly parallels to the dash line $t^{-\frac{1}{3}}$,
which suggests that the energy evolves in time as the power law $Ct^\alpha$ with $\alpha\approx-\frac{1}{3}$.
The (b) presents the power law for the growth of the interface height $h(t)$, which is defined by
$$h(t)=\bigg(\frac{1}{|\Omega|}\int_\Omega u^2(x,y,t)\,\d x\d y\bigg)^\frac{1}{2}.$$
Again, the height curve is plotted in log-to-log scale.
The growth of the interface height approximately obeys the power law $Ct^\beta$ with $\beta\approx\frac{1}{3}$,
which is consistent with the existing works (see, e.g., \cite{QiSuZh12,QiZhTa11,XuTa06}).

\begin{figure}[!htp]
\centering
\subfigure[energy evolution]{\includegraphics[scale=0.9]{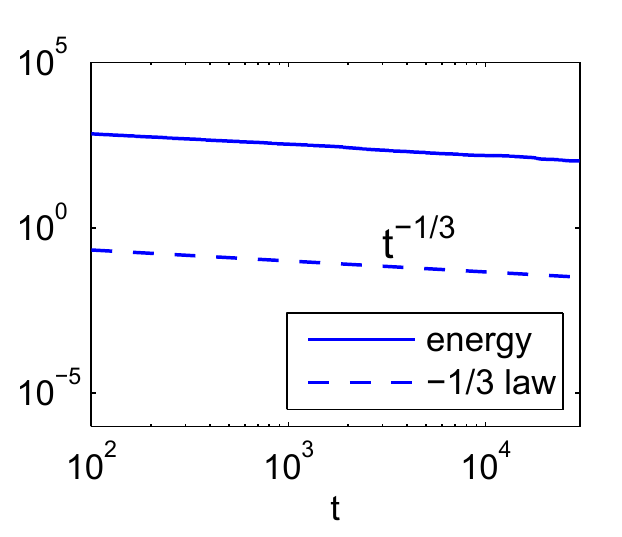}}
\subfigure[height growth]{\includegraphics[scale=0.9]{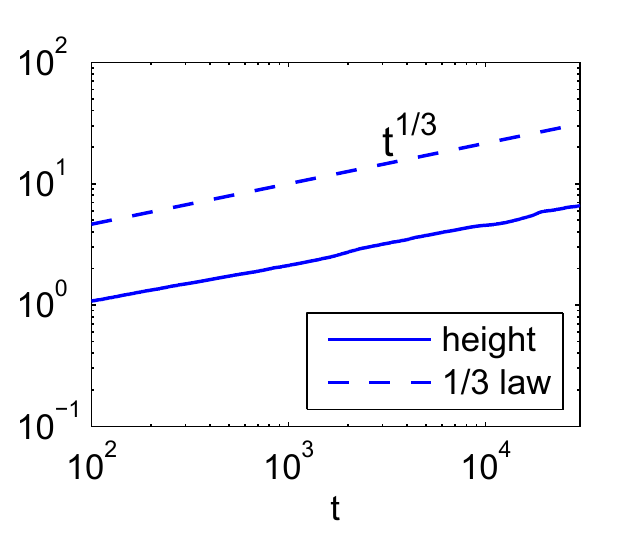}}
\caption{Example \ref{numerical_eg2}: The power law for the evolution of the energy and the interface height.}\label{fig_evolution}
\end{figure}

\section{Conclusions}

In this work, we investigate the error estimate of a fast explicit operator splitting method
for a nonlinear fourth-order diffusion equation modeling epitaxial growth of thin films.
The convergence order $\mathcal{O}(\tau^2+h^4)$ in discrete $L^2$-norm is proved theoretically and verified numerically.
For the nonlinear subproblem, we construct a 25-point center-difference scheme in space and use the third-order explicit SSP-RK scheme in time.
Since fewer points are involved in the scheme at each node compared to the 33-point center-difference scheme presented in \cite{ChKuQuTa15},
the restriction for the stability may be reduced.
We carry out several numerical experiments to verify the efficiency of the derived algorithm
and present some results for small $\delta$'s with the time step $\tau=\delta/10$.
Furthermore, we find numerically the coarsening exponents for the energy evolution and the height growth are $-1/3$ and $1/3$, respectively.

\section*{Acknowledgment}

We are grateful to Prof.~Zhengru Zhang of Beijing Normal University for many valuable comments.
The research of Zhonghua Qiao is partially supported by the Hong Kong Research Council GRF
grants 202112, 15302214, 509213 and NSFC/RGC Joint Research Scheme N\_HKBU204/12.
Hui Zhang is partially supported by NSFC/RGC Joint Research Scheme No.~11261160486,
NSFC grant No.~11471046 and the Ministry of Education Program for New Century Excellent Talents Project NCET-12-0053.

\end{document}